\documentclass[a4paper,12pt]{article}
\usepackage[top=2.5cm,bottom=2.5cm,left=2.5cm,right=2.5cm]{geometry}
\usepackage[margin=1cm,%
font=small,%
format=hang,%
labelsep=period,%
labelfont=bf]{caption}
\usepackage{amsmath}
\usepackage{amsfonts}
\usepackage{amssymb}
\usepackage{cite}
\usepackage{amsthm}
\usepackage{makeidx}
\usepackage{graphicx}
\usepackage{mathrsfs,xcolor}
\usepackage{enumerate}
\usepackage{blkarray}
\usepackage{bm}
\usepackage{setspace}
\usepackage{multicol,multirow}
\usepackage{csquotes}

\usepackage{authblk}

\DeclareMathOperator{\supp}{supp}
\DeclareMathOperator{\Ima}{Im}

\DeclareMathOperator{\spn}{span}

\numberwithin{table}{section}
\numberwithin{equation}{section}

\theoremstyle{plain}
\newtheorem{theorem}{Theorem}
\newtheorem{lemma}{Lemma}
\newtheorem{proposition}{Proposition}
\newtheorem{corollary}{Corollary}

\theoremstyle{definition}
\newtheorem{definition}{Definition}
\newtheorem{example}{Example}
\newtheorem{remark}{Remark}

\usepackage{authblk}
\author[1,2]{ \textbf{Honeylou F. Farinas}}
 \author[2,3,4,5]{\textbf{Eduardo R. Mendoza}}
 \author[2,*]{\textbf{Angelyn R. Lao}}

\affil[1]{\small \textit{Department of Mathematics, Mariano Marcos State University, Ilocos Norte, 2906 Philippines}}
\affil[2]{\textit{Mathematics and Statistics Department, De La Salle University, Manila  0922, Philippines }}
\affil[3]{\textit{Center for Natural Sciences and Environmental Research, De La Salle University, Manila  0922, Philippines }}
\affil[4]{\textit{Max Planck Institute of Biochemistry, Martinsried near Munich, Germany}}
\affil[5]{\textit{Faculty of Physics, Ludwig Maximilian University, Munich 80539, Germany}}
\affil[*]{Corresponding author: \texttt{angelyn.lao@dlsu.edu.ph}}

\title{\vspace{3.5cm}\textbf{Chemical reaction network decompositions and realizations of S-systems}}

\begin{document}
\maketitle
\thispagestyle{empty}
\begin{abstract}
This paper presents novel decomposition classes of chemical reaction networks (CRNs) derived from S-system kinetics. Based on the network decomposition theory initiated by Feinberg in 1987, we introduce the concept of incidence independent decompositions and develop the theory of $\mathscr{C}$- and $\mathscr{C}^*$- decompositions which partition the set of complexes and the set of nonzero complexes respectively, including their structure theorems in terms of linkage classes. Analogous to Feinberg's independent decomposition, we demonstrate the important relationship between sets of complex balance equilibria for an incidence independent decomposition of weakly reversible subnetworks for any kinetics. We show that the  $\mathscr{C}^*$-decompositions are also incidence independent. We also introduce in this paper a new realization for an S-system that is analyzed using a newly defined class of species coverable CRNs. This led to the extension of the deficiency formula and characterization of fundamental decompositions of species decomposable reaction networks. 
\newline

\end{abstract}
\baselineskip=0.30in

\section{Introduction}
\label{intro}
S-systems consist of ordinary differential equations (ODEs) of the form
$$\dfrac{dX_i}{dt}=\alpha_i\displaystyle\prod_{j=1}^{m}X_j^{g_{ij}}-
\beta_i\displaystyle\prod_{j=1}^{m}X_j^{h_{ij}},\; i\in\{1,\dots,m\}$$
where $\alpha_i, \beta_i$ are nonnegative and the exponents $g_{ij},h_{ij}$ are arbitrary real numbers. For convenience, we assume that the variables are restricted to positive real values. They form a special class of power law dynamical systems, which are called Generalized Mass Action (GMA) systems in Biochemical Systems Theory (BST). $S$-systems were introduced in M. Savageau's seminal work in 1969 \cite{SAVA1969-1,SAVA1969-2,VOIT2000} and have been extensively applied in modeling complex biochemical systems in many fields, see e.g. Voit's review \cite{VOIT2012}. Various authors have studied chemical kinetic systems (CKS) which are realizations, i.e. dynamically equivalent or have the identical set of ODEs as the dynamical system \cite{HOJA1972, MURE2012, MURE2014, AJMM2015, AJLM2017}.

This paper presents novel concepts and results on decompositions of chemical reaction networks (CRN) that we derived from
\begin{itemize}
\item the analysis of a kinetic system realization of an $S$-system model of \textit{Mycobacterium tuberculosis (Mtb)} by Magombedze and Mulder \cite{MAMU2013, FML2020}, and
\item the study of kinetic system realizations of $S$-systems which are dominant subsystems of GMA systems in design space theory \cite{SCFTS2009}.
\end{itemize} 
 
In the first part of the paper, we introduce the concept of incidence independent decompositions of a CRN, which complements the independence property defined by M. Feinberg in 1987 \cite{FEIN1987}. A basic property of an incidence independent decomposition is the inequality: $\delta \geq \delta_1 + ...+ \delta_k$, where $\delta$ and $\delta_i$ denote the deficiency of the network and the $i^{th}$ subnetwork respectively.

We identify the important subset of $\mathscr{C}$-decompositions, which are those generated by partitions of the reaction set which are also partitions of the set of complexes. The best known example of a $\mathscr{C}$-decomposition is the set of linkage classes of a CRN. We provide a characterization of $\mathscr{C}$-decompositions in terms of linkage classes.

Feinberg demonstrated the importance of independent decompositions by stating the relationship of the sets of positive equilibria of the network and those of the subnetworks for any kinetics. We derive the analogous result for incidence independent decompositions of weakly reversible subnetworks and the sets of complex balanced equilibria. We conclude the first part with the study of $\mathscr{C^*}$-decompositions, which partition the nonzero complexes of the network and show that these are also incidence independent.

In the second part of the paper, we introduce a new realization for an $S$-system, which we call the subnetwork realization. This concept is motivated by studies of design spaces, where $S$-subsystems of a GMA system are used to analyze the system's behavior. To analyze the subnetwork realization and its predecessor, now called the independent realization of an $S$-system, we introduce the class of species coverable CRNs. We then extend the deficiency formula of Arceo et al. \cite{AJMM2015} and the characterization of fundamental decompositions by Hernandez et al. \cite{HMR2020} for the independent realization of an $S$-system to any species decomposable CRN.

The main new results of the paper are:
\begin{itemize}
\item the characterization of the structure of $\mathscr{C}$-decompositions (Theorem \ref{thm:Structure thm for C-decomposition});
\item the relationship between sets of complex balanced equilibria for an incidence independent decomposition of weakly reversible subnetworks for any kinetics (Theorems \ref{theorem:decomposition and complex balanced equilibria} and \ref{theorem:complex balanced equilibria});
\item the characterization of the structure of $\mathscr{C^*}$-decompositions (Theorem \ref{thm:STC*}) and its corollary (incidence independence of $\mathscr{C^*}$-decompositions); and
\item the extensions of the deficiency formula and characterization of fundamental decompositions to species decomposable reaction networks (Theorem \ref{thm:species decomposable CRNs}).
\end{itemize}

The paper is organized as follows: Section 2 collects the fundamental concepts and results on CRNs and kinetic systems needed in the latter sections. The result on incidence independent decompositions, $\mathscr{C}$-decompositions and $\mathscr{C^*}$-decompositions are derived and discussed in Section 3. Section 4 introduces the realizations of $S$-systems, species coverable and species decomposable CRNs. The final section provides a summary and outlook.


\section{Fundamentals of chemical reaction networks and kinetic systems}
\label{sec:2}

We recall the necessary concepts of chemical reaction networks and the mathematical notations used throughout the paper adopted from the papers \cite{AJMM2015}, \cite{FEIN1987} and \cite{FLRM2019}. 

We begin with the definition of a chemical reaction network.
\begin{definition}
A \textbf{chemical reaction network} is a triple $\mathscr{N}=(\mathscr{S},\mathscr{C},\mathscr{R})$ of three non-empty finite sets:
\begin{enumerate}
\item A set of $m$ \textbf{species} $\mathscr{S}$,
\item A set $\mathscr{C}$ of $n$ \textbf{complexes}, which are non-negative integer linear combinations of the species, and
\item A set $\mathscr{R} \subseteq \mathscr{C} \times \mathscr{C}$ of $n$ \textbf{reactions} such that
\begin{itemize}
\item $(i,i) \notin \mathscr{R}$ for all $ i \in \mathscr{C}$, and
\item for each $i \in \mathscr{C}$, there exists a $j \in \mathscr{C}$ such that $(i,j) \in \mathscr{R}$ or $(j,i) \in \mathscr{R}.$
\end{itemize}
\end{enumerate}

\end{definition}

Two useful maps are associated with each reaction:

\begin{definition}
The \textbf{reactant map} $\rho: \mathscr{R} \rightarrow \mathscr{C}$ maps a reaction to its reactant complex while the \textbf{product map} $\pi: \mathscr{R} \rightarrow \mathscr{C}$ maps it to its product complex. We denote $|\rho (\mathscr{R})|$ with $n_r$, i.e., the number of reactant complexes. 
\end{definition}

Connectivity concepts in Digraph Theory apply to CRNs, but have slightly differing names. A connected component is traditionally called a \textbf{linkage class}, denoted by $\mathscr{L}$, in CRNT. A subset of a linkage class where any two elements are connected by a directed path in each direction is known as a \textbf{strong linkage class}. If there is no reaction from a complex in the strong linkage class to a complex outside the same strong linkage class, then we have a \textbf{terminal strong linkage class}. We denote the number of linkage classes with $l$, that of the strong linkage classes with $sl$ and that of terminal strong linkage classes with $t$. Clearly, $sl \geq t \geq l.$ A CRN is said to be \textbf{weakly reversible} if $sl=l$ , and it is said to be \textbf{$t$-minimal} if $t=l$.

Many features of CRNs can be examined by working in terms of finite dimensional spaces $\mathbb{R}^{\mathscr{S}}, \mathbb{R}^{\mathscr{C}}, \mathbb{R}^{\mathscr{R}},$ which are referred to as species space, complex space and reaction space, respectively. We can view a complex $j \in \mathscr{C}$ as a vector in $\mathbb{R}^{\mathscr{C}}$ (called \textit{complex vector}) by writing $j = \sum _{s \in \mathscr{S}} j_s s,$ where $j_s$ is the stoichiometric coefficient of species $s$.

\begin{definition}
The \textbf{reaction vectors} of a CRN $(\mathscr{S}, \mathscr{C}, \mathscr{R})$ are the members of the set $\{j-i \in \mathbb{R}^{\mathscr{S}} | (i,j) \in \mathscr{R}\}.$ The \textbf{stoichiometric subspace} $S$ of the CRN is the linear subspace of $\mathbb{R}^{\mathscr{S}}$ defined by 
	$$S: \spn \{j-i \in \mathbb{R}^{\mathscr{S}} | (i,j) \in \mathscr{R}\}.$$
The \textbf{rank} of the CRN, $s$, is defined as $s=\dim S.$
\end{definition}

\begin{definition}
The \textbf{incidence map} $I_a: \mathbb{R}^{\mathscr{R}} \rightarrow \mathbb{R}^{\mathscr{C}}$ is defined as follows. For $f: \mathscr{R} \rightarrow \mathbb{R}$, then $I_a(f)(v) = - f(a)$ and $f(a)$ if $v = \rho(a)$ and $v = \pi(a)$, respectively, and are $0$ otherwise.
\end{definition}

\noindent Equivalently, it maps the basis vector $\omega_a$ to  $\omega_{v'} -  \omega_v$ if $a: v \rightarrow v'$.\\

It is clearly a linear map, and its matrix representation (with respect to the standard bases $\omega_a$, $\omega_{v}$) is called the \textbf{incidence matrix}, which can be described as 
\begin{center}
\[
 (I_a)_{i,j} = 
  \begin{cases} 
   -1 & \text{if } \rho(a_j) = v_i, \\
   1       & \text{if } \pi(a_j) = v_i,\\
   0		& \text{otherwise}.
  \end{cases}
\]
\end{center}
\noindent Note that in most digraph theory books, the incidence matrix is set as $-I_a$.\\
An important result of digraph theory regarding the incidence matrix is the following:

\begin{proposition}
\label{prop:IncidenceMatrix02}
Let $\mathscr{N}=(\mathscr{S},\mathscr{C},\mathscr{R})$ be a CRN. Denote by $I$ the incidence matrix of the directed graph $(\mathscr{C},\mathscr{R})$. Then rank $I = n -l$, where $n$ is the number of complexes and $l$ is the number of linkage classes of a CRN.
\end{proposition}

A non-negative integer, called the deficiency, can be associated to each CRN. This number has been the center of many studies in CRNT due to its relevance in the dynamic behavior of the system.

\begin{definition}
The \textbf{deficiency} of a CRN is the integer $\delta = n-l-s.$ 
\end{definition}
 We can also define the deficiency not only for the whole network, but also for each linkage class $\mathscr{L}_i$. The \textbf{deficiency of linkage class $\mathscr{L}_i$} (denoted by $\delta_i$) is defined by the formula: $\delta_i=n_i-l_i-s_i=n_i-1-s_i.$
 
\begin{definition}
The \textbf{reactant subspace} $R$ is the linear space in $\mathbb{R}^\mathscr{S}$ generated by the reactant complexes. Its dimension, $\dim R$ denoted by $q$, is called the \textbf{reactant rank} of the network. Meanwhile, the \textbf{reactant deficiency} $\delta_p$ is the difference between the number of reactant complexes and the reactant rank, i.e., $\delta_p = n_r -q.$
\end{definition}

We now introduce the fundamentals of chemical kinetic systems. We begin with the general definitions of kinetics from Feliu and Wiuf \cite{FW2012}:

\begin{definition}
A \textbf{kinetics} for a CRN $(\mathscr{S}, \mathscr{C}, \mathscr{R})$ is an assignment of a rate function $K_j: \Omega_K \rightarrow \mathbb{R}_\geq$ to each reaction $r_j \in \mathscr{R}$, where $\Omega_K$ is a set such that $\mathbb{R}^\mathscr{S}_> \subseteq \Omega_K \subseteq \mathbb{R}^\mathscr{S}_\geq$, $c\wedge d \in \Omega_K$ whenever $c,d \in \Omega_K,$ and 
$$K_j(c) \geq 0, \quad \forall c \in \Omega_K.$$
A kinetics for a network $\mathscr{N}$ is denoted by $\displaystyle{K=(K_1,K_2,...,K_r):\Omega_K \to {\mathbb{R}}^{\mathscr{R}}_{\geq}}$. The pair $(\mathscr{N}, K)$ is called the \textbf{chemical kinetic system} (CKS).
\end{definition}

\noindent
In the definition, $c \wedge d$ is the bivector of $c$ and $d$ in the exterior algebra of $\mathbb{R}^\mathscr{S}.$ We add the definition relevant to our context:

\begin{definition}
A chemical kinetics is a kinetics $K$ satisfying the positivity condition: for each reaction $j:y\rightarrow y', K_j(c)>0$ if and only if the $\supp y\subset\supp c$.
\end{definition}

Once a kinetics is associated with a CRN, we can determine the rate at which the concentration of each species evolves at composition $c$.

\begin{definition}
The \textbf{species formation rate function (SFRF)}  of a CKS is the vector field  $f(x) = NK (x) = \displaystyle\sum_{y\rightarrow y'}K_{y\rightarrow y'}(x) (y'- y).$  The equation $\frac{dx}{dt} = f(x)$ is the \textbf{ODE or dynamical system} of the CKS.  A zero of $f$ is an element $c$ of $\mathbb R^{\mathscr{S}}$ such that $f(c) = 0.$ A zero of $f$ is an \textbf{equilibrium or steady state} of the ODE system. 
\end{definition}

\begin{definition}
The \textbf{set of positive equilibria} of a CKS $(\mathscr{N},K)$ is given by $$E_+(\mathscr{N},K)=\{x \in \mathbb{R}_>^{\mathscr{S}}| f(x)=0\}.$$
\end{definition}

\begin{definition}
A positive vector $c$ in $\mathbb{R}^{\mathscr{S}}$ is called \textbf{complex balanced} (CB) if $K(x)$
is contained in $\ker I_a$. Further, if $c$ is a positive equilibrium then we call it a complex
balanced equilibrium. We denote by $Z_+(\mathscr{N},K)$ the \textbf{set of complex balanced equilibria} of a CKS system $(\mathscr{N},K).$
\end{definition}

\section{Incidence independent decompositions of chemical reaction networks}
\label{sec:3}

Decomposition theory of CRNs was initiated by M. Feinberg in his 1987 review \cite{FEIN1987}, where he introduced the general definition of a decomposition and listed some of its basic properties. He identified the important subclass of independent decompositions and stated the relationship between positive equilibria sets of the network and subnetworks for such decompositions. We first review his results in the more general context of coverings and unions of CRNs.

\subsection{Coverings, unions and independent decompositions of Chemical Reaction Networks}
\label{sec:3.1}

In this section, we introduce the concept of a covering, a minor generalization of a decomposition, and relate it to the unions of CRNs.

\begin{definition}
Let $\mathscr{N}=(\mathscr{S},\mathscr{C},\mathscr{R})$ be a CRN. A \textbf{covering of $\mathscr{N}$} is a set of subsets of $\mathscr{R}_i$ whose union is $\mathscr{R}$. A covering is called a \textbf{decomposition of $\mathscr{N}$} if the sets $\mathscr{R}_i$ form a partition of $\mathscr{R}$. 
\end{definition}

Clearly, each $\mathscr{R}_i$ defines a subnetwork $\mathscr{N}_i$ of $\mathscr{N}$, namely $\mathscr{C}_i$  consisting of all complexes occurring in $\mathscr{R}_i$  and $\mathscr{S}_i$ consisting of all species occurring in $\mathscr{C}_i$.

In \cite{GHMS2019}, the concept of the union of chemical reaction networks was introduced as follows: 

\begin{definition}
The \textbf{union} of reaction networks $\mathscr{N}_1=(\mathscr{S}_1,\mathscr{C}_1,\mathscr{R}_1)$ and $\mathscr{N}_2=(\mathscr{S}_2,\mathscr{C}_2,\mathscr{R}_2)$ is 
$$\mathscr{N}_1 \cup \mathscr{N}_2=(\mathscr{S}_1 \cup \mathscr{S}_2,\mathscr{C}_1 \cup \mathscr{C}_2,\mathscr{R}_1 \cup \mathscr{R}_2.$$
The union of finitely many reaction networks $\mathscr{N}_i$ is defined similarly.
\end{definition} 

If $\mathscr{N}$ is the union of subnetworks $\mathscr{N}_i$, then clearly the reaction sets $\mathscr{R}_i$ form a covering of $\mathscr{N}.$  Conversely, under \enquote{normal condition} (as defined by the CRN properties assumed in the following Proposition), we have:

\begin{proposition}
 Let $\mathscr{N} = (\mathscr{S},\mathscr{C},\mathscr{R})$ be a CRN with the following properties:
\begin{enumerate}[i)]
\item Each complex $y$ occurs in at least one reaction, i.e. there are no isolated complexes.
\item Each species occurs in at least one complex, i.e. $S$ is the union of the $\supp y$, with $y \in \mathscr{C}.$
\end{enumerate}
If $\mathscr\{{R}_i\}$ is a covering of $\mathscr{N}$, then $\mathscr{N}$ is the union of subnetworks $\{\mathscr{N}_i\}$ defined by the covering.
\end{proposition}

\begin{proof}
By assumption, the union of the reaction subsets is $\mathscr{R}$. Since any complex of $\mathscr{N}$ must occur in at least one reaction, say in $\mathscr{R}_i$, then it is contained in $\mathscr{N}_i$. Hence the union of the $\mathscr{C}_i$ is $\mathscr{C}$.  Since any species must occur in at least one complex, then it is contained in an $\mathscr{S}_i$. Hence the union of the $\mathscr{S}_i$ is $\mathscr{S}$.  Thus, the whole CRN is the union of the subnetworks of the covering.
\end{proof}

This proposition formally justifies the term \enquote{covering of a network}. In the rest of the paper, we assume that the CRNs have these two typical properties, and hence, there is a correspondence between coverings of a network and its representation as a union of subnetworks. As in the case of decompositions, we have:
\begin{proposition}
If $\mathscr\{{R}_i\}$  is a network covering, then 
\begin{enumerate}[i)]
\item $S = S_1 + ... + S_k.$
\item $s \leq s_1 + ... + s_k$, where $s$, $s_i$ are the dimensions of the subspaces
\end{enumerate}
\end{proposition}

Feinberg identified the important subclass of independent decomposition:

\begin{definition}
A \textbf{decomposition is independent} if $S$ is the direct sum of the subnetworks´ stoichiometric subspaces $S_i$.  Equivalently, $s = s_1 + ... + s_k$.
\end{definition}

In \cite{FLRM2019}, Fortun et al. derived a basic property of independent decompositions:
\begin{proposition} If $\mathscr{N}  = \mathscr{N}_1 \cup... \cup \mathscr{N}_k$, is an independent decomposition, then $\delta \leq \delta_1 +... + \delta_k.$
\end{proposition}

When studying decompositions of a network, a useful relationship is given by set-theoretic containment:

\begin{definition}
 If $\mathcal{P}=\{\mathcal{P}_i\}$ and $\mathcal{P}'=\{\mathcal{P}'_j\}$ are partitions of a set, then $\mathcal{P}$ is a \textbf{refinement} of $\mathcal{P}'$ if each $\mathcal{P}_i$ is contained in (exactly) one $\mathcal{P}'_j$. 
\end{definition}
 
It is easy to show that this property is equivalent to each $\mathcal{P}'$ being the disjoint union of some $\mathcal{P}_i$'s. We also say the $\mathcal{P}$ is \textbf{finer} than $\mathcal{P}'$, $\mathcal{P}'$ is \textbf{coarser} than $\mathcal{P}$ and $\mathcal{P}'$ is a \textbf{coarsening} of $\mathcal{P}$.
 
\begin{proposition}
\label{prop: independent coarsening}
If a decomposition is independent, then any coarsening of the decomposition is independent.
\end{proposition} 
 
\begin{proof}
 Suppose $x$ is in the intersection of the stoichiometric subspaces of two subnetworks of a coarsening. Since each stoichiometric subspace is the direct sum of subspaces from the independent refinement, then the $x$ is the sum of elements from each subnetwork. It follows that $x=0$.
\end{proof}

\subsection{Incidence independent decompositions and their basic properties}

We now introduce the new concept of an incidence independent decomposition, which naturally complements the independence property. Our starting point is the following basic observation:

\begin{proposition}  
If $\mathscr\{{R}_i\}$  is a network covering, then 
\begin{enumerate}[i)]
\item $\Ima I_a = \Ima I_{a,1} + ... + \Ima I_{a,k}$
\item $n -  l \leq (n_1-l_1)+... + (n_k-l_k),$ where $n-l, n_i-l_i$ are the dimensions of the subspaces
\end{enumerate}
\end{proposition}

The analogous concept to independence is the following:

\begin{definition}
A decomposition \{$\mathscr{N}_1,...,\mathscr{N}_k\}$ of a CRN is \textbf{incidence independent} if and only if the image of the incidence map of $\mathscr{N}$ is the direct sum of the images of the incidence maps of the subnetworks. 
\end{definition} 

Since the direct sum property of the images is equivalent to the dimension of the image of $I_a$ which is equal to the sum of the dimensions of the subnetwork incident map images, an equivalent formulation is the following equality: 
\begin{equation}
\label{eqn:incidence independence}
n-l = \sum(n_i-l_i).
\end{equation}

\begin{example}
The linkage classes form the primary example of an incidence independent decomposition, since $n=\sum n_i$ and $l=\sum l_i$. In fact, the linkage class decompositions belong to the important subclass of $\mathscr{C}$-decompositions discussed in the next section. 
\end{example}

We have the following analogue of the result of Fortun et al. \cite{FLRM2019}, a property familiar from linkage classes:

\begin{proposition}
\label{prop:deficiency of incidence independent}
For an incidence independent decomposition $\mathscr{N}=\mathscr{N}_1 \cup...\cup\mathscr{N}_k$, then $\delta \geq \delta_1+...+\delta_k$.
\end{proposition}

\begin{proof}
Since for any decomposition, $s \leq \sum s_i$, subtracting the LHS from $n-l$ and the RHS from $\sum n_i -\sum l_i$ delivers the claim.
\end{proof}

We have the following proposition which is analogous to Proposition \ref{prop: independent coarsening}:

\begin{proposition}
\label{prop:incidence independent coarsening}
If a decomposition is incidence independent, then any coarsening of the decomposition is incidence independent.
\end{proposition} 

The proof has the same argumentation as in Proposition \ref{prop: independent coarsening} now applied to the image of the incidence map instead of the stoichiometric subspace.

\begin{definition}
A decomposition is \textbf{bi-independent} if it is both independent and incidence independent.
\end{definition}

The independent linkage class (ILC) property of linkage classes is the best known example of a bi-independent decomposition.

The relationship between bi-independence of a decomposition $\mathscr{N}  = \mathscr{N}_1 \cup ... \cup \mathscr{N}_k$, and $\delta = \delta_1 + ...+ \delta_k$ is expressed in the following Proposition:

\begin{proposition}
A decomposition $\mathscr{N}=\mathscr{N}_1 \cup ... \cup \mathscr{N}_k$ is independent or incidence independent and $\sum \delta_i=\delta$ iff $\mathscr{N}=\mathscr{N}_1 \cup ... \cup \mathscr{N}_k$ is bi-independent.
\end{proposition}
 
\begin{proof}
 \enquote{$\Leftarrow$} follows from combining the deficiency inequalities for independence and incidence independence. 
 For \enquote{$\Rightarrow$}, if  $\delta = \delta_1 +... + \delta_k$  and the decomposition is independent (or incidence independent), then adding $s = s_1 +... + s_k$ (or $n-l =(n_1-l_1)+ ... + (n_k-l_k$)) to the deficiency equality yields bi-independence.
 \end{proof}
 
\begin{corollary} 
Let $\mathscr{N}=(\mathscr{S}, \mathscr{C}, \mathscr{R})$ be a CRN, and $\mathscr{N}_i= (\mathscr{S}_i, \mathscr{C}_i, \mathscr{R}_i), i=1,...,k$ be the deficiency zero subnetworks of a decomposition which is independent or incidence independent. Then $\delta=0$ iff the decomposition is bi-independent.
\end{corollary}

\begin{proof}
$\delta=0=\sum \delta_i$ and independence or incidence independence $\Leftrightarrow$ bi-independence according to the previous proposition.
\end{proof}

\subsection{The subset of $\mathscr{C}$-decompositions}
We now study an important subset of incident independent decompositions, the $\mathscr{C}$-decompositions. Recall that for a decomposition $\mathscr{N}  = \mathscr{N}_1 \cup ... \cup \mathscr{N}_k,$ $\mathscr{C}_i$ denotes the set of complexes occurring in the reaction set $\mathscr{R}_i$, i.e. $y \in \mathscr{C}_i$ iff there is a reaction $r \in \mathscr{R}_i$  such that $y$ is the reactant or product of the reaction $r$.

\begin{definition}
A decomposition $\mathscr{N}=\mathscr{N}_1\cup \mathscr{N}_2\cup...\cup \mathscr{N}_k$ with $\mathscr{N}_i=(\mathscr{S}_i,\mathscr{C}_i,\mathscr{R}_i)$ is a \textbf{$\mathscr{C}$-decomposition} if $\mathscr{C}_i \cap \mathscr{C}_j =\emptyset $ for $i \neq j$.
\end{definition}

A $\mathscr{C}$-decomposition partitions not only the set of reactions but also the set of complexes. The primary example of a $\mathscr{C}$-decomposition are the linkage classes. Linkage classes, in fact, essentially determine the structure of a $\mathscr{C}$-decomposition. We present this Structure Theorem for a $\mathscr{C}$-decomposition: 

\begin{theorem} (Structure Theorem for $\mathscr{C}$-decomposition)
\label{thm:Structure thm for C-decomposition}
Let $\mathscr{L}_1,...,\mathscr{L}_l$ be the linkage classes of a network $\mathscr{N}$. A decomposition $\mathscr{N}=\mathscr{N}_1 \cup \mathscr{N}_2\cup...\cup \mathscr{N}_k$ is a $\mathscr{C}$-decomposition if and only if each $\mathscr{N}_i$ is the union of linkage classes and each linkage class is contained in only one $\mathscr{N}_i$. In other words, the linkage class decomposition is a refinement of $\mathscr{N}$.
\end{theorem}

\begin{proof}
Clearly, if the linkage classes form a refinement of $\mathscr{N}$, then $\mathscr{N}$ is a $\mathscr{C}$-decomposition. To see the converse, let $\mathscr{N}_i=(\mathscr{S}_i,\mathscr{C}_i,\mathscr{R}_i)$ and $\mathscr{L}_j=(\mathscr{S}_{\mathscr{L}_j},\mathscr{C}_{\mathscr{L}_j},\mathscr{R}_{\mathscr{L}_j})$ where $\mathscr{R}_i$ is the union (taken over $j$) of $(\mathscr{R}_i \cap \mathscr{R}_{\mathscr{L}_j})$. We only need to show that each non-empty intersection is equal to $\mathscr{R}_{\mathscr{L}_j}$, ( i.e., $\mathscr{R}_{\mathscr{L}_j}=\mathscr{R}_i \cap \mathscr{R}_{\mathscr{L}_j}$) to imply that each linkage class is contained in only one $\mathscr{N}_i$. If the linkage class $\mathscr{L}_j$ has only one reaction then $\mathscr{R}_i \cap \mathscr{R}_{\mathscr{L}_j}=\mathscr{R}_{\mathscr{L}_j}$. If the linkage class $\mathscr{L}_j$ has at least two reactions, then there is an adjacent reaction to each reaction, whose reactant complex or product complex is common with the first reaction. If this adjacent reaction belongs to a different subnetwork, then there exists a complex which is common to two different subnetworks. This would contradict that $\mathscr{N}$ partitions the set of complexes. Hence, all reactions of the linkage class lie in the intersection with $\mathscr{R}_i$.
\end{proof}

\begin{corollary}
For a $\mathscr{C}$-decomposition $\mathscr{N}=\mathscr{N}_1 \cup \mathscr{N}_2\cup...\cup\mathscr{N}_k$, $k \leq l$. 
\end{corollary}

\begin{proof}
If $\mathscr{N}$ is decomposed according to linkage classes, then $\mathscr{N}_i=\mathscr{L}_i$. Thus, $k=l$. If each $\mathscr{N}_i$ is the union of linkage classes, then the number of subnetworks is less than the number of linkage classes. Hence, $k<l$.
\end{proof}

\begin{corollary}
\label{cor:CInciIndep}
Any $\mathscr{C}$-decomposition is incident independent.
\end{corollary}

\begin{proof}
The linkage class decomposition is incidence independent, so any coarsening of it is also incidence independent after Proposition \ref{prop: independent coarsening}. 
\end{proof}

\begin{example}
In \cite{AJLM2018}, Arceo et al. introduced the subnetwork $\mathscr{N}_S$ of $S$-complexes of a CRN $\mathscr{N}$ and used this in characterizing the classification of CRNs based on the intersection of $R$ and $S$.
\end{example}

We recall the relevant definition from \cite{AJLM2018}:

\begin{definition}
An \textbf{$S$-complex} of a CRN is a complex which, as a vector in $\mathbb{R}^{\mathscr{S}}$, is contained in the stoichiometric subspace $S$. We denote the subset of $S$-complexes in $\mathscr{C}$ with $\mathscr{C}_S$.
\end{definition}  

The CRN classification based on the intersection of the reactant and stoichiometric subspaces is summarized in Figure \ref{fig:Network classes}.

\begin{figure}[h!]
\begin{center}
    \includegraphics[width=0.8\textwidth]{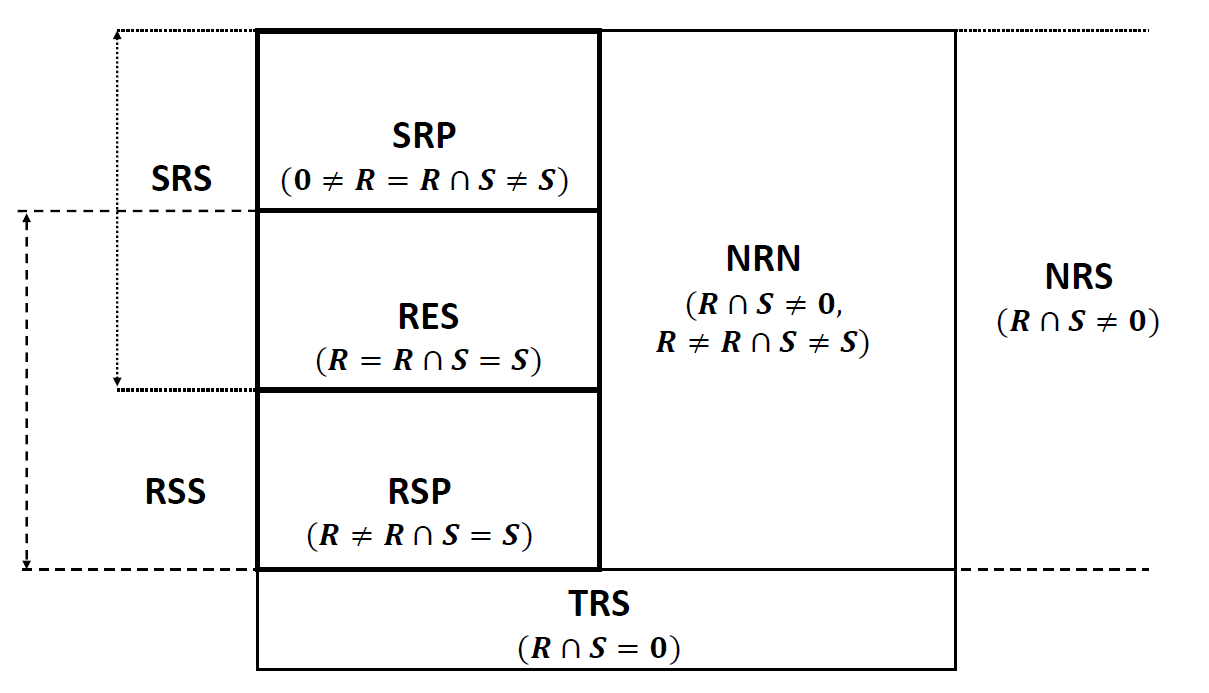}
\end{center}
\caption{An overview of the network classes.}
\label{fig:Network classes}
\end{figure}

The following Theorem provides the relationship between $\mathscr{N}_S$ and the network classes:

\begin{theorem}(Arceo et al. \cite{AJLM2018})
\label{thm:relationship between N_S and network classes}
Let $Y$ be the map of complexes of a network $\mathscr{N}$ with subnetwork $\mathscr{N}_S$ of $S$-complexes.
\begin{enumerate}[i)]
\item $\mathscr{N}$ is $SRS \Leftrightarrow \Ima Y = S \Leftrightarrow c=s.$ Furthermore, $\mathscr{N}$ is $SRS \Rightarrow \mathscr{N}=\mathscr{N}_S.$
\item $\mathscr{N}$ is $RSS \Leftrightarrow \Ima Y = R \Leftrightarrow c=q.$ Furthermore, $\mathscr{N}$ is $RSS \Rightarrow$ either $\mathscr{N}=\mathscr{N}_S(RES)$ or $\mathscr{N} \neq \mathscr{N}_S(RSP).$
\item $\mathscr{N}$ is $TRS \Leftrightarrow \Ima Y$ is a direct sum of  $R$ ans $S \Leftrightarrow c=q+s.$ Furthermore, $\mathscr{N}$ is $TRS \Rightarrow \mathscr{N} \neq \mathscr{N}_S$ and, if $\mathscr{N}$ has no inflow reaction, $\mathscr{N}_S = \phi.$
\item $\mathscr{N}$ is $NRN \Rightarrow c < q+s < 2c.$ Furthermore, $\mathscr{N}$ is $NRN \Rightarrow \mathscr{N} \neq \mathscr{N}_S.$
\end{enumerate}
\end{theorem}
 
Arceo et al. \cite{AJLM2018} showed that $\mathscr{N}_S$ had the distinctive property of being a union of linkage classes. We now show that this derives from the fact that $\mathscr{N}_S$ is part of a $\mathscr{C}$-decomposition which we call the $S$-decomposition. We have the following Lemma:

\begin{lemma}
For any reaction $y \rightarrow y'$, either both complexes $y$ and  $y'$ are in $S$ or none of them.
\end{lemma}

\begin{proof} 
If $y$ is in $S$, then $y'= (y' \cdot y) + y$ is also in $S$. Similarly, $y'$ in $S$ implies $y = y' - (y'- y)$ is in $S.$
\end{proof}

\begin{definition}
 A reaction $y \rightarrow y'$ is an $S$-reaction if both complexes $y$ and  $y'$ are in $S$. $\textbf{R}_S$ denotes the set of $S$-reactions, the subnetwork-complement decomposition it induces is called the $S$-decomposition.
\end{definition}

The $S$-decomposition is clearly a $\mathscr{C}$-decomposition, and hence its subnetworks $\mathscr{N}_S$ and $\mathscr{N}_{NS}$ are unions of linkage classes. Theorem \ref{thm:relationship between N_S and network classes} provides a good example showing that in NRN network, $\mathscr{N}_S$ is non-empty.

We also obtain a new characterization of the ILC property:

\begin{corollary}
A network has independent linkage classes if and only if every $\mathscr{C}$-decomposition is independent.
\end{corollary}

\begin{proof}
If a network has independent linkage classes then the stoichiometric subspace $S$ is the direct sum of the stoichiometric subspaces of the linkage classes. Grouping the summands  according to the unions of the linkage classes for the subnetworks of a $\mathscr{C}$-decomposition provides $S$ as the direct sum of the subnetworks. Hence, every $\mathscr{C}$-decomposition of a network is also independent. For the converse, since every $\mathscr{C}$-decomposition of a network is independent and the linkage class decomposition is also a $\mathscr{C}$-decomposition, it follows that the linkage classes are independent.
\end{proof}

\noindent
If a network has dependent linkage class, it may fail to have an independent $\mathscr{C}$-decomposition, as the following example shows:
\begin{example}
Consider the CRN with reactions $X_1 \rightarrow 2X_1 +X_2$ and $X_2 \rightarrow 2X_2 + X_1$, it has $\delta=1$. The only non-trivial decomposition is the linkage class decomposition, where the deficiency of the two linkage classes is 0. Clearly, the linkage class decomposition is dependent. In particular, it has no independent $\mathscr{C}$-decomposition. 
\end{example}

\subsection{Incidence independent decompositions and complex balanced equilibria}

Feinberg established the following basic relation between an independent decomposition and the set of positive equilibria of a kinetics on the network:

\begin{theorem}(Feinberg Decomposition Theorem \cite{FEIN1987}) 
Let $P(\mathscr{R})=\{\mathscr{R}_1,\mathscr{R}_2,...,\mathscr{R}_k\}$ be a partition of a CRN $\mathscr{N}$ and let $K \in \mathscr{K}(\mathscr{N}).$ If $\mathscr{N}=\mathscr{N}_1+\mathscr{N}_2+...+\mathscr{N}_k$ is a network decomposition of $P(\mathscr{R})$ and $E_+(\mathscr{N}_i, K_i)=\{x \in \mathbb{R}_+^m|N_iK_i(x)=0\}$ then 
$$E_+(\mathscr{N}_1,K_1) \cap E_+(\mathscr{N}_2, K_2) \cap ... \cap E_+(\mathscr{N}_k, K_k) \subseteq E_+(\mathscr{N}, K).$$ 
If the network decomposition is independent, then equality holds.
\end{theorem}

Our main result in this section is the analogue of Feinberg's 1987 result for incidence independent decompositions and complex balanced equilibria:

\begin{theorem}
\label{theorem:decomposition and complex balanced equilibria}
Let $\mathscr{N}=(\mathscr{S}, \mathscr{C}, \mathscr{R})$ be a CRN and $\mathscr{N}_i=(\mathscr{S}_i, \mathscr{C}_i, \mathscr{R}_i), i=1,...,k$ be the subnetworks of a decomposition. Let $K$ be any kinetics and $Z_+(\mathscr{N},K), Z_+(\mathscr{N}_i,K_i)$ and $E_+(\mathscr{N}_i, K_i)$ be as defined above. Then:
\begin{enumerate}[i)]
\item $\bigcap Z_+(\mathscr{N}_i, K_i) \subset Z_+(\mathscr{N},K)$ \\
If the decomposition is incidence independent, then
\item $Z_+(\mathscr{N}, K)=\bigcap Z_+(\mathscr{N}_i, K_i)$ 
\item $Z_+(\mathscr{N}, K)\neq \emptyset$ then $Z_+(\mathscr{N}_i, K_i)\neq \emptyset$ for each $i$.
\end{enumerate}
\end{theorem}

\begin{proof}
$\mathscr{N}, \mathscr{N}_i$ are not assumed weakly reversible, so $Z_+(\mathscr{N}),\quad Z_+(\mathscr{N}_i)$ may be empty. Let $\mathscr{R}_i$ be the reaction subset defining $\mathscr{N},$ and $K_i:R^{\mathscr{S}} \rightarrow R^{\mathscr{R}_i}$ is given by $pr_i \circ K,$ where $pr_i$ is the projection from $R^\mathscr{R}$ to $R^{\mathscr{R}_i}.$ Furthermore set $I_{a,i}:=res_{\mathscr{R}_i}I_a.$ Then $I_a K(x)=\sum_i I_{a,i}K_i(x).$ Clearly, in i) if the LHS is empty, there is nothing to prove. If $x \in \bigcap Z_+(\mathscr{N}_i,K_i)$, then $I_{a,i}K_i(x)=0$ for each $I$, and hence their sum $I_aK(x)=0$, or $x \in Z_+(\mathscr{N})$.

To show ii), again if the LHS is empty, then we are done. If $x \in Z_+(\mathscr{N})$, then $I_aK(x)=0=\sum I_{a,i}K_i(x)=0=\sum 0.$ Since the decomposition is incidence independent, it follows that $I_{a,i}K(x)=0$ for each $I$, or $Z_+(\mathscr{N},K) \subset \bigcap Z_+(\mathscr{N}_i,K_i).$ Note that in this case, $\mathscr{N}$ and $\mathscr{N}_i$ are necessarily weakly reversible. Equality then follows from i). iii) follows directly from ii).  
\end{proof}

The converse statement of Theorem \ref{theorem:decomposition and complex balanced equilibria} (iii) holds for a subset of incidence independent decompositions with any kinetics. This is a significant contrast to the case of independent decompositions where the converse statement is known only for a few restricted kinetics such as MAK and PL-TIK, which are power law kinetics with zero kinetic reactant deficiency \cite{TMJ2019}. To show this part of our second main result, we need the following proposition:

\begin{proposition} (Boros \cite{BOROS2013})
\label{prop:Image A}
Let $l, m \in \mathbb{Z}_+$ and $n_1,n_2,...,n_l \in \mathbb{Z}_+.$ Let $A_j \in \mathbb{R}^{n_j} (j \in \{1,2,...,l\}).$ Assume that $\{x \in \mathbb{R}^m|A_j \cdot x = b_j\} \neq \emptyset$ for all $j \in \{1, 2, ..., l\}$ and $$\Ima [A_1^\top ,  A_2^\top , ... , A_l^\top]=\Ima A_1^\top \oplus \Ima A_2^\top \oplus ... \oplus \Ima A_l^\top.$$ Then
$$\bigcap \{x \in \mathbb{R}^m| A_j \cdot x=b_j\} \neq \emptyset.$$ 
\end{proposition}

\begin{theorem}
\label{theorem:complex balanced equilibria}
Let $\mathscr{N}=\mathscr{N}_1 \cup \mathscr{N}_2 \cup ... \cup \mathscr{N}_k$ be a weakly reversible $\mathscr{C}$-decomposition of a chemical kinetic system $(\mathscr{N},K).$ Then, if $Z_+(\mathscr{N}_i,K) \neq \emptyset$ for each subnetwork, $Z_+(\mathscr{N}, K) \neq \emptyset.$
\end{theorem}

\begin{proof}
We first consider the case when the $\mathscr{C}$-decomposition is the linkage class decomposition. It is well known that the incidence map $I_a$ has a block matrix decomposition, after the complex rows are arranged as $\mathscr{C}_1,...,\mathscr{C}_l$ and the reaction columns as $\mathscr{R}_1,...\mathscr{R}_l$ respectively:
\begin{equation*}
I_a = 
\begin{bmatrix}
I_{a,1} & \quad & 0 \\
\quad & \ddots & \quad  \\
0 & \quad & I_{a,l} 
\end{bmatrix}
\end{equation*}
Furthermore, a complex balanced equilibrium of $(\mathscr{N},K)$ and $(\mathscr{N}_i,K)$ is the image of $K(x)$ contained in $ker I_a$ and $ker I_{a,i},$ respectively. Since the $\mathscr{C}$-decomposition is incidence independent, $\Ima I_a$ is the direct sum of the images of the incidence maps of the subnetworks. In view of the block matrix description of $I_a$, we also obtain that $\Ima I_a^\top$ is the direct sum of the images of the transposed maps on the subnetworks. Hence, Proposition \ref{prop:Image A} is applicable and we have:
\begin{equation}
\label{eqn:Z+}
Z_+(\mathscr{N},K)=\bigcap Z_+(\mathscr{N}_i,K)\neq \emptyset.
\end{equation}

In Theorem \ref{thm:Structure thm for C-decomposition}, it is shown that any $\mathscr{C}$-decomposition is generated by a coarsening of the partition of the reaction set into the reaction sets of linkage classes. This implies that the incidence map of the subnetworks are groupings of the terms in Equation \ref{eqn:Z+} and hence the intersection is taken over the same sets of equilibria, which proves the claim.
\end{proof}

\begin{remark}
The two previous Theorems were derived for the special case of the linkage class decomposition and poly-PL kinetics in \cite{TMMNJ2020}.
\end{remark}

The following tables emphasize the analogous characteristics of independent and incidence independent decompositions:

\begin{table}[h!]
\begin{center}   
\begin{tabular}{ll}
\hline\noalign{\smallskip}
Independent decomposition: characteristic & Reference/Comment \\
\noalign{\smallskip}\hline\noalign{\smallskip}
Definition: $S$ is direct sum of subnetwork $S_i$			& Feinberg 1987 \cite{FEIN1987}  \\
Deficiency relationship: $\delta \leq \delta_1 + ... + \delta_k$				& Fortun et al. 2018 \cite{FLRM2018} \\
Equilibria sets for any kinetics: $E_+(\mathscr{N},K)=\cap E_+(\mathscr{N}_i,K_i)$				& Feinberg 1987 \cite{FEIN1987} \\
Coarsening invariance: any coarsening of an 			& this paper \\
independent decomposition is also independent &   \\
\noalign{\smallskip}\hline
\end{tabular}
\end{center}
\end{table}

\begin{table}[h!]
\begin{center}   
\begin{tabular}{ll}
\hline\noalign{\smallskip}
Incidence Independent decomposition: characteristic & Reference/Comment \\
\noalign{\smallskip}\hline\noalign{\smallskip}
Definition: $\Ima I_a$ is direct sum of subnetwork $I_{a,i}$			& this paper  \\
Deficiency relationship: $\delta \geq \delta_1 + ... + \delta_k$				& this paper /linkage class \\
    &  case: Feinberg 1987 \cite{FEIN1987}\\
CB Equilibria sets for weakly reversible decomposition and 				& this paper /linkage class  \\
any kinetics: $Z_+(\mathscr{N},K)=\cap Z_+(\mathscr{N}_i,K_i)$ & and PY-RDK case: \\
  & Talabis et al. 2018 \cite{TMMNJ2020}\\
Coarsening invariance: any coarsening of an incidence 			& this paper\\ 
independent decomposition is also incident independent &  \\
\noalign{\smallskip}\hline
\end{tabular}
\end{center}
\end{table}

\subsection{$\mathscr{C}^*$-decompositions}

In their S-system model of \textit{Mtb} gene regulation, Magombedze and Mulder \cite{MAMU2013} introduced three subsystems, which in the CRN representation generated a decomposition into three subnetworks, whose sets of non-zero complexes were pairwise disjoint. This led us to define the set of $\mathscr{C}^*$-decompositions of chemical reaction networks and study their properties.

\begin{definition}
A decomposition $\mathscr{N}=\mathscr{N}_1 \cup \mathscr{N}_2\cup...\cup\mathscr{N}_k$ with $\mathscr{N}_i=(\mathscr{S}_i,\mathscr{C}_i,\mathscr{R}_i)$ is a \textbf{$\mathscr{C}^*$-decomposition} if $\mathscr{C}_i^* \cap \mathscr{C}_j^* =\emptyset $ for $i \neq j$ where $\mathscr{C}_i^*$ and $\mathscr{C}_j^*$ are the non-zero complexes in $\mathscr{C}_i$ and $\mathscr{C}_j$, respectively. 
\end{definition}

Clearly, the set of $\mathscr{C}$-decompositions is contained in this set.

\begin{remark}

Gross et al. \cite{GHMS2019} call the union of two networks \enquote{complex-disjoint} if the intersection of their sets of complexes is contained in $\{0\}.$ If the covering defined by the union is a decomposition, then this construct is identical with a $\mathscr{C}^*$-decomposition.  However, we find their terminology somewhat confusing since the zero complex is a bona fide complex. In our view, the \enquote{complex-disjoint} decompositions are the $\mathscr{C}$-decompositions.
\end{remark}

The following Theorem describes the general structure of $\mathscr{C}^*$-decompositions.

\begin{theorem}(Structure Theorem for  $\mathscr{C}^*$-decomposition)
\label{thm:STC*}
Let $\mathscr{N}_1 \cup \mathscr{N}_2 \cup ... \cup \mathscr{N}_k$ be a $\mathscr{C}^*$-decomposition and $\mathscr{L}_0$ and $\mathscr{L}_{0,i}$ be the linkage classes of $\mathscr{N}$ and $\mathscr{N}_i$ containing the zero complex (note $\mathscr{L}_{0,i}$, is empty if $\mathscr{N}_i$ does not contain the zero complex). Then
\begin{enumerate}[i)]
\item the $\mathscr{L}_{0,i}$ form a $\mathscr{C}^*$-decomposition of $\mathscr{L}_0$
\item the (non-empty) $\mathscr{N}_i \setminus \mathscr{L}_{0,i}$ form a $\mathscr{C}$-decomposition of $\mathscr{N} \setminus \mathscr{L}_0$ 
\end{enumerate} 
\end{theorem}

\begin{proof}
To prove (i), we need to show that each non-zero complex of $\mathscr{L}_0$ is contained in only one subnetwork $\mathscr{N}_i$. If there is only one subnetwork $\mathscr{N}_i$ containing the zero complex then we are done. If there are at least two subnetworks containing the zero complex then $\mathscr{L}_0$ has at least two non-zero complexes connected to the zero complex. Otherwise, if there would only be one complex then $\mathscr{N}_i$ is not a $\mathscr{C}^*$-decomposition of $\mathscr{N}$, a contradiction. Now, if one of these non-zero complexes belongs to different subnetworks, this would contradict that $\mathscr{N}$ partitions the non-zero complexes. Hence, all the non-zero complexes of $\mathscr{L}_0$ is contained in only one $\mathscr{N}_i$ and $\mathscr{L}_{0,1} \cup \mathscr{L}_{0,2} \cup...\cup \mathscr{L}_{0,j} = \mathscr{L}_0$ for $j \leq k$.

To prove (ii), it suffices to show that the intersection of the set of complexes in $\mathscr{N}_i \setminus \mathscr{L}_0$ is empty. The set of complexes in $\mathscr{N} \setminus \mathscr{L}_0$ are all non-zero and $\mathscr{N} \setminus \mathscr{L}_0 = (\mathscr{N}_1 \cup ...\cup \mathscr{N}_k) \setminus \mathscr{L}_0$. From (i), we have $\mathscr{L}_{0,1} \cup \mathscr{L}_{0,2} \cup...\cup \mathscr{L}_{0,j} = \mathscr{L}_0$ for $j \leq k$. Thus, $\mathscr{N} \setminus \mathscr{L}_0 = \mathscr{N}_1 \setminus \mathscr{L}_{0,1} \cup ... \cup \mathscr{N}_k \setminus \mathscr{L}_{0,k}$ where $\mathscr{L}_{0,k}$ is empty if $\mathscr{N}_k$ does not contain the zero complex. Since $\mathscr{N}_i$ is a $\mathscr{C}^*$-decomposition of $\mathscr{N}$, the intersection of the set of complexes in $\mathscr{N}_i \setminus \mathscr{L}_{0,i}$ is empty.
\end{proof}

We will now use the previous result to prove the incidence independence of $\mathscr{C}^*$-decompositions. The number $k(0)$ of subnetworks containing the zero complex turns out to be a useful tool for this. If $k(0)=0$, then the network does not contain the zero complex, hence the set of $\mathscr{C}^*$-decompositions is simply the set of $\mathscr{C}$-decompositions. For positive values, it can be used to formulate a convenient criterion for incidence independence:

\begin{corollary}
\label{cor:C*InciIndep}
Any $\mathscr{C}^*$-decomposition is incidence independent.
\end{corollary}

\begin{proof}
We recall from Equation (\ref{eqn:incidence independence}) that
$$n-l = \sum (n_i -l_i)$$ where $n_i$ and $l_i$ are number of complexes and linkage classes in the subnetwork $\mathscr{N}_i$. 

Suppose $k(0) > 0$. If $k(0)$ subnetworks contain the zero complex then the number of complexes $n=n^*+1$, where $n^*$ is the number of non-zero complexes. For any $\mathscr{N}_i$ of the $k(0)$ subnetworks, for the corresponding numbers, we also have $n_i = n_i^* +1$, for all others, $n_i=n_i^*.$ Thus, Equation (\ref{eqn:incidence independence}) becomes
$$n^*+1-l=\sum n_i^* +k(0) - \sum l_i.$$
Since $n^*=\sum n^*_i$, we obtain $\sum l_i - l = k(0)-1.$
\end{proof}

According to the Structure Theorem, the linkage classes of $\mathscr{N}$ consist of $\mathscr{L}_0$ and such which contain only non-zero complexes, $\mathscr{L}_1, ... ,\mathscr{L}_{l -1}.$ Each of these remaining linkage classes must however be contained in exactly one subnetwork, and hence a linkage class of that subnetwork. Conversely, each linkage class in a subnetwork with only non-zero complexes is a linkage class of the whole network. Therefore $l =  1 + \sum l_i + \sum(l_i - 1),$ where the first sum is over all subnetworks not containing $0$ and the second over all subnetworks containing $0$. Therefore $l =  1 + \sum l_i - k(0),$ which is the formulated criterion for incidence independence above. 

\section{Species coverable CRNs and S-system realizations}

In this Section, we introduce a new realization of an S-system in order to enable a CRNT approach in the context of recent developments in BST on phenotype-oriented modeling based on design spaces. This realization will be defined in the framework of total realizations of BST systems in Section \ref{subsec:Total realization of BST}. In order to study the new and old S-system realizations in a semantically consistent manner, we introduce the class of species coverable CRNs and its subset of species decomposable CRNs. Our main result in this Section, Theorem \ref{thm:species decomposable CRNs}, corrects and extends a deficiency formula by Arceo et al. \cite{AJMM2015} and a result by Hernandez et al. \cite{HMR2020} to species decomposable CRNs.

\subsection{S-systems and their realizations}
We first review the current realization introduced by Arceo et al. in \cite{AJMM2015} and \cite{AJLM2017}.

\subsubsection{Current realization of an S-system}
To any given S-system, Arceo et al. associated the biochemical map (see Figure \ref{fig:S system}) and obtained CRNs which they called stoichiometric and total representations \cite{AJMM2015}. To obtain a realization, i.e. a dynamically equivalent kinetic system, they constructed in \cite{AJLM2017} the embedded network of the total representation given by the subsets of dependent species and the full reaction set. This realization, called the embedded representation, has the advantage of using the minimum number of species needed and corresponded to the BST practice of \enquote{lumping} the independent variables with the rate constants for each power law term.

\begin{figure}[h!]
\begin{center}
    \includegraphics[width=0.5\textwidth]{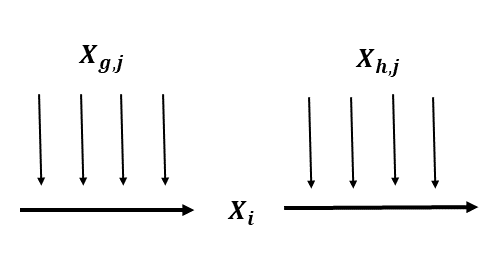}
\end{center}
\caption{Biochemical map of an $S$-system.}
\label{fig:S system}
\end{figure}

Let $R_i$ and $P_i$ be the sets of the variables regulating the input and output arrow for $X_i$ (as in Figure \ref{fig:S system}) and $\underline{R_i}, \underline{P_i}$ be the sums of the elements in $R_i$ and $P_i$ respectively. In this realization, the reaction subsets $\{\underline{R_i} \rightarrow X_i + \underline{R_i}, X_i + \underline{P_i} \rightarrow \underline{P_i}\}$ have the following property:

\begin{proposition}
The reaction sets form an independent decomposition of the embedded representation.
\end{proposition}

\begin{proof}
i) We have to show that if $i\neq j$, then the intersection of the reaction sets is empty.
Suppose that the sets $\{\mathscr{R}'_i\}$ do not form a partition of the reaction set $\mathscr{R}'$. Then there exists two sets $\mathscr{R}'_i$ and $\mathscr{R}'_j$, where $i \neq j$, that has a common reaction. We consider the following cases: a) two inflow reactions coincide and b) an inflow reaction coincides with an outflow reaction. The remaining cases involve converse reactions and hence follow similarly.

We let $\mathscr{R}'_i=\{R_i \rightarrow X_i+R_i, X_i + P_i \rightarrow P_i\}$ and $\mathscr{R}'_j=\{R_j \rightarrow X_j+R_j, X_j + P_j \rightarrow P_j\}$. We denote the subvectors of $R_i, R_j, P_i$ and $P_j$ as  $V_i, V_j, W_i$ and $W_j$, respectively. We set $V_i=(a_1,...,a_m)$ and $W_j=(b_1,...,b_m)$. In connection to the $a$ elements of $V_i$, the two input reactions in $\mathscr{R}'_i$ and $\mathscr{R}'_j$ coincide thus $R_i=R_j$ and $X_i +R_i=X_j +R_j$. This implies that $V_i=V_j$ and $X_i + V_i=X_j+V_i$ or $(a_1,...,a_i+1,...,a_m)=(a_1,...,a_j+1,...,a_m)$. Since $i \neq j$, $a_i+1=a_i$ and $a_j=a_j+1$, a contradiction.

As for the $b$ elements of $W_j$, we assume that an inflow reaction in $\mathscr{R}'_i$ coincides with an $\mathscr{R}'_j$. Then $R_i=X_j+P_j$ and $X_i +R_i=P_j$. Thus, we have $V_i=X_j+W_j$ or $(a_1,...,a_j,...,a_m)=(b_1,...,b_j+1,...,b_m)$. This implies that $a_i=b_i$ and $a_j=b_j+1$. Similarly, $V_i+X_i=W_j$ implies that $a_i+1=b_i$ and $a_j=b_j$. Since $i \neq j$, $a_i=b_i=a_i+1$ and $b_j=a_j=b_j+1$, a contradiction.

ii) Note that the stoichiometric subspace $S_i$ of each of the $m$ subnetworks $\mathscr{Ri}'_i$ of the species decomposition is $\{X_i\}$. Thus, the rank of each $\mathscr{Ri}'_i$ is 1. Since there are $m$ subnetworks and the rank of an $S$-system is m, $s=m=s_1+...+s_m$ and this implies independence. 
\end{proof}

\subsubsection{Total realization of a BST system}
\label{subsec:Total realization of BST}

We now introduce an additional realization for any BST system, in particular, any GMA system given by a biochemical map:

\begin{definition}
The total realization of a BST system is the total representation with an additional outflow reaction for each independent variable together with the power law kinetics specified by the kinetic order matrix.
\end{definition}

The additional outflow reaction for each independent variable enables the corresponding ODE $\frac{dX_i}{dt} = 0$ to be solvable in all cases, hence, resulting in a realization of the GMA system. Clearly, the total realization has the same sets of species and complexes as the total representation, but has $m_I$ additional reactions $(m_I =$ number of independent species).

\begin{example}
 The total realization derived from the total representation of an S-system (as reviewed in Section \ref{subsec:Total realization of BST}) will be denoted as the independent realization of the S-system, in order to distinguish it from the new realizations to be introduced in Section \ref{subsec:subnetwork realization of S}.
 \end{example}
 
Note that the embedded networks formed by the subset of dependent species and all reactions of total representations and the total realizations are identical.  For consistency in terminology, we will henceforth denote the embedded representations as \enquote{embedded realizations} of the BST systems.

\subsection{Subnetwork realization of an S-system}
\label{subsec:subnetwork realization of S}
In BST´s design space theory \cite{SCFTS2009}, a phenotype-oriented analysis of the behavior of a biochemical system described by a GMA model is conducted by constructing S-subsystems and identifying parameter regions where the S-subsystem is dominant, i.e. its values determine the behavior. The initial step in design space theory selects for each dependent variable, a positive (activating) and a negative (degrading) term from its ODE. In the GMA system´s biochemical map, this amounts to selecting an input arrow and an output arrow for each dependent variable. In the CRN of the total realization, we obtain a subnetwork defined as follows: 

\begin{definition}
The \textbf{subnetwork realization} of an S-subsystem of a GMA system is the total realization of the union of the subnetworks generated by the reaction pairs $\{ X_{i,\rho} + S_i \rightarrow X_i + S_i, X_i + P_i \rightarrow X_{i,\pi} + P_i\}$ for each dependent species $X_i$, where $S_i$, $P_i$ being the sums of the corresponding regulatory species of the input and output arrows. The sets of reaction pairs form the species covering of the subnetwork realization.
\end{definition}

As the simple example, $0\rightarrow X \rightarrow Y \rightarrow Z \rightarrow 0$ shows, this species covering is not necessarily a decomposition, since several inflow reactions coincide with previous outflow reactions: in fact, there are only $4$ reactions instead of $6 = 2 \times 3.$
  
\begin{remark}
The terminology \enquote{independent realization} of an S-system highlights the fact it is represented independently of any containing network/system. Fortuitously, the subsets of reaction pairs also form an independent decomposition of the independent realization.
\end{remark}

\subsection{Species coverable and species decomposable CRNs}
Despite its contrasting semantic interpretation to the subnetwork realization, we readily observe that formally the independent realization is a special case  of the subnetwork realization: if we set $X_{i,\rho} = X_{i,\pi} = 0,$ we obtain the independent realization. In order to have a consistent semantic framework, we abstract a level further and introduce a class of CRNs containing the networks of both realizations. We then use this class to formulate and derive common properties.

\begin{definition}
\label{def:speciescoverable}
\begin{enumerate}[i)]
\item A CRN with species set $\mathscr{S} = \{X_1,... ,X_m\}$ is \textbf{species coverable} if for each $X_i$, there are species  $X_{i,\rho}, X_{i,\pi} \in (\mathscr{S} \setminus \{X_i\}) \cup \{0\}$ and subsets  $R_i, P_i$ of $\mathscr{S}$  with $\underline{R_i}, \underline{P_i}$ be the sums of their respective elements such that $\mathscr{R}$ is the union of $\mathscr{R}_1,... ,\mathscr{R}_m$ with $\mathscr{R}_i = \{ X_{i,\rho} + \underline{ R_i} \rightarrow X_i + \underline{R_i}, X_i + \underline{P_i} \rightarrow X_{i,\pi} + \underline{P_i}\}.$  The $\mathscr{R}_i$'s  form the species covering of the network.
\item A species is called independent if $X_{i,\rho} = X_{i,\pi} = 0$ and $R_i = P_i =\phi$ (hence by convention, $\underline{R_i} = \underline{P_i} = 0).$ Otherwise, it is a dependent species. A species is reversible if $X_{i,\rho} = X_{i,\pi}$ (hence all independent species are reversible).
\item A species coverable CRN is \textbf{species decomposable} if the species covering is an independent decomposition.
\end{enumerate} 
\end{definition}

\begin{example}
\begin{enumerate}
\item The CRN of the subnetwork realization of an S-system is species coverable
\item The CRN of the independent realization of an S-system is species decomposable.
\end{enumerate}
\end{example}

We have the following main result for species decomposable CRNs:

\begin{theorem}
\label{thm:species decomposable CRNs}
Let $\mathscr{N}$ be a species decomposable CRN. Then 
\begin{enumerate} [i)]
\item $\delta \leq m-m_{rev},$ where $m_{rev}$ is the number of reversible species. If the species decomposition is bi-independent, then $\delta = m-m_{rev}.$
\item the fundamental decomposition of $\mathscr{N}$   is the species decomposition.
\end{enumerate}
\end{theorem}

\begin{proof}
\begin{enumerate}[i)]
\item Since the species decomposition is independent, $s = s_1 +...+ s_m$. Since $s_i \geq 1$, this implies that  $s_i = 1$ and $s = m$.  For a reversible species $X_i$, we have $\delta_i = 2-1-1 = 0$. For an irreversible species, we have either $\delta_i = 4-2 -1 = 1$ or $\delta_i = 3-1 - 1 = 1$. Hence $\sum \delta_i = m- m_{rev}.$ Hence, $\delta \leq m - m_{rev}$ since the decomposition is independent. If it is also incidence independent, then $=$ holds.
\item We denote the inflow reaction in $\mathscr{R}_i$ with $r_{-i}$, and the corresponding basis vector with $\omega_i$ and $\omega_{-i}$, respectively. We set $m':=m-m_{rev}$, and as remarked above, since the species decomposition is independent, the network is open (i.e. $s = m$ and $\dim S_i = 1$ for each $i$. Hence for any orientation, the cardinality $= 2m- m_{rev}$, and $\dim Ker L_\mathscr{O} = m- m_{rev}$. For each irreversible species $X_i$, we can write  $$X_i - X_{i,\rho} =  \lambda_i (X_{i,\pi} - X_i) \Leftrightarrow (1 + \lambda_i) X_i = X_{i,\rho} + \lambda_iX_{i,\pi}.$$ We claim that the vectors $\omega_i - \lambda_i\omega_{-I}$ lie in $$Ker L_\mathscr{O}: L_\mathscr{O} (\omega_i - \lambda_i\omega_i) = X_i -X_{i,\rho} -\lambda_i(X_{i,\pi}-X_i)=0.$$ They are linearly independent and hence form a basis. On the other hand, the $m$ vectors $\lambda_i\omega_i + \omega_I , \chi_j$ with $i=1,2,...,m',$ and $j=1,2,...,m_{rev}$ and $\chi_i$ the reaction from a reversible pair included in the orientation, form a basis for $Ker^{\perp}L_{\mathscr{O}}.$ From the $\mathscr{F}$-decomposition definition, the reactions $\omega_i$ and $\omega_{-i}$ are equivalent, $i=1,2,...,m.$ If $k \neq i,\langle \omega_k - \alpha \omega_i,\omega_i+\omega_{-i} \rangle = -\alpha,$ so that if $\alpha$ is nonzero, then the $k$-th inflow reaction is not equivalent. Similarly, the $k$-th outflow reaction is not equivalent. Hence, the $\mathscr{F}$-equivalence classes are precisely the $\mathscr{R}_i's.$ 
\end{enumerate}
\end{proof}

\begin{remark}
Since in Proposition 3.19 of \cite{HMR2020}, $\lambda_i=-1$ for all $l$, there is a typo in the proof: instead of \enquote{$... \omega_i + \omega_{-l},\chi_j ...$} it should read  \enquote{$... -\omega_i+\omega_{-l},\chi_j...$}.
\end{remark}


Finally, we note the following new formulation of Theorem 3 in \cite{FML2020}.

\begin{theorem}
Any species coverable CRN with two or more dependent species is discordant.
\end{theorem}

\begin{remark} For positive equilibria, we have the following hierarchy of subsets and CRN classes in which they may exist:  detailed balanced (DB) equilibria (reversible CRNs) $\subset$ complex balanced (CB) equilibria (weakly reversible CRNs) $\subset$ positive equilibria (any CRN). This hierarchy corresponds to balance of the level of reactions (DB), to balance on the level of complexes (CB) and balance on the level of species.
\end{remark}

With the introduction of species decomposable CRNs, we obtain the following (restricted) conceptual hierarchy:
\begin{itemize}
\item (level of reactions), for any CRN, any decomposition determines (is in fact equal to) a partition of the set of reaction
\item (level of complexes), for any CRN, any $\mathscr{C}$-decomposition, in addition, determines a partition of the set of complexes
\item (level of species) for any species decomposable CRN, the fundamental decomposition (= species decomposition) determines a partition of the set of species (into singletons)
\end{itemize}

The restriction is of course that the species level is valid only for a small class of CRNs.


\section{Conclusions and outlook}
This paper presents novel decomposition classes of chemical reaction networks (CRNs) derived from S-system kinetics:

We introduced the concept of coverings, a minor generalization of a decomposition, and relate it to the unions of CRNs. A covering is called a \textbf{decomposition of $\mathscr{N}$} if the sets $\mathscr{R}_i$ form a partition of $\mathscr{R}$. Given the network covering properties, we introduced the basic property of incidence independent decompositions of a CRN, which complements the independence property defined by M. Feinberg in 1987 \cite{FEIN1987}. We have shown in Proposition~\ref{prop:deficiency of incidence independent} that for an incidence independent decomposition $\mathscr{N}=\mathscr{N}_1 \cup...\cup\mathscr{N}_k$, $\delta \geq \delta_1+...+\delta_k$.

In this paper, we have presented the following new results:
\begin{itemize}
\item The theory of $\mathscr{C}$- and $\mathscr{C}^*$- decompositions which partition the set of complexes and the set of nonzero complexes respectively, including their structure theorems in terms of linkage classes (shown in Theorem~\ref{thm:Structure thm for C-decomposition} and Theorem~\ref{thm:STC*}, respectively). We have shown that $\mathscr{C}$ and $\mathscr{C}^*$- decompositions are both incident independent (shown in Corollary~\ref{cor:CInciIndep} and Corollary~\ref{cor:C*InciIndep}, respectively). 
\item Analogous to Feinberg's independent decomposition, we demonstrate the important relationship between sets of complex balance equilibria for an incidence independent decomposition of weakly reversible subnetworks for any kinetics (Theorems \ref{theorem:decomposition and complex balanced equilibria} and \ref{theorem:complex balanced equilibria}).
\item We have introduced a new realization for an S-system that is analyzed using a newly defined class of species coverable CRNs (see Definition~\ref{def:speciescoverable}). This led to the extension of the deficiency formula and characterization of fundamental decompositions to species decomposable reaction theorem (Theorem \ref{thm:species decomposable CRNs}).
\end{itemize}


\vspace{10pt}

\noindent \textbf{Acknowledgments}. HFF acknowledges the support of the Commission on Higher Education (CHED), Philippines for the CHED-SEGS Scholarship Grant. ARL held research fellowship from De La Salle University and would like to acknowledge the support of De La Salle University’s Research Coordination Office.


\end{document}